\documentclass[a4, 12pt, reqno]{article}
\usepackage{amscd, amsmath, amssymb, amsfonts, amsthm}
\usepackage[arrow, matrix]{xy}
\usepackage{graphicx}
\usepackage{color}
\usepackage[utf8]{inputenc}
\usepackage[english]{babel}
\theoremstyle{plain}
\newtheorem{theorem}{Theorem}[section]
\newtheorem{lemma}{Lemma}[section]

\theoremstyle{definition}

\begin{document}

\title{On the maximum CEI of graphs with parameters}
\author{Fazal Hayat\footnote{E-mail: fhayatmaths@gmail.com}\\
School of Mathematical Sciences,  South China Normal University, \\
 Guangzhou 510631,  PR China}

 \date{}
\maketitle

\begin{abstract}
The connective eccentricity index (CEI) of a graph $G$ is defined as $\xi^{ce}(G)=\sum_{v \in V(G)}\frac{d_G(v)}{\varepsilon_G(v)}$, where $d_G(v)$ is the degree of $v$ and $\varepsilon_G(v)$ is the eccentricity of $v$. In this paper, we characterize  the unique graphs with  maximum CEI from three classes of graphs: the $n$-vertex graphs with fixed connectivity and diameter, the $n$-vertex graphs  with fixed connectivity and independence number, and the $n$-vertex graphs with fixed connectivity and minimum degree.  \\ \\
{\bf Key Words}: connective eccentricity  index,   connectivity, diameter, minimum degree, independent number.\\\\
{\bf 2010 Mathematics Subject Classification:}  05C07; 05C12; 05C35
\end{abstract}

\section{Introduction}
All graphs considered in this paper are simple and connected. Let $G$ be graph on $n$ vertices with vertex set $V(G)$ and edge set $E(G)$.  For  $v \in V(G)$, let $N_G(v)$ be the set of all neighbors of $v$ in $G$. The degree of $v \in V(G)$, denoted by $d_G(v)$,  is the cardinality of  $N_G(v)$. The maximum and minimum degree of $G$ is denoted by $\Delta(G)$ and $\delta(G)$, respectively.  The graph formed from $G$ by deleting any vertex $v \in V(G)$ (resp. edge $uv \in E(G)$) is denoted by $G-v$  (resp. $G-uv$). Similarly, the graph formed from $G$  by adding an edge  $uv$ is denoted by $G+uv$, where $u$ and $v$ are  non-adjacent vertices of $G$.  For a vertex subset $A$ of $V(G)$, denote by $G[A]$ the subgraph induced by $A$. The distance between  vertices $u$ and $v$ of $G$, denoted by $d_G(u,v)$, is the length of a shortest path connecting $u$ and $v$ in $G$.  For $v\in V(G)$, the eccentricity of $v$ in $G$, denoted $\varepsilon_G(v)$, is the  maximum distance from $v$ to all other vertices of $G$. The diameter of a graph $G$  is the maximum eccentricity of all vertices in $G$. As usual, by $S_n$ and $P_n$ we denote the star and path on $n$ vertices, respectively.

A subset $A$ of $V(G)$ is called vertex cut of $G$ if $G-A$ is disconnected. The minimum size of $A$ such that $G-A$ is disconnected or has exactly one vertex is called connectivity of $G$, and denote by $k(G)$. A graph is $k$-connected if its connectivity is at least $k$.
A subset $I$ of  $V(G)$ is called an independent set of $G$ if $I$ contains pairwise non-adjacent vertices. The independence number of $G$ denoted by $\alpha(G)$, is the maximum number of vertices of  independent sets of $G$.
The join of two disjoint graphs $M_1$ and $M_2$ is denoted by $M_1 \vee M_2$ is the graph formed from $M_1 \cup M_2$ by adding the edges $\{e=xy : x\in V(M_1), y\in V(M_2) \}$. Let $M_1 \vee M_2 \vee  M_3 \vee \dots \vee M_t=(M_1\vee M_2)\cup (M_2\vee M_3)\cup \dots \cup (M_{t-1}\vee M_t)$. The sequential join of $t$ disjoint copies of a graph $G$ is denoted by $[t]G$, and the union of $m$ disjoint copies of a graph $G$ is denoted by $mG$.

Topological indices are numbers reflecting certain structural features of a molecule that are derived from its molecular graph. They  are used in theoretical chemistry for design of chemical compounds with given physicochemical properties or given pharmacological and biological activities.

The eccentric connectivity index of a graph $G$ is a topological index based on degree and eccentricity, defined as
\[
\xi^c (G)=\sum_{v \in V(G)}d_G(v)\varepsilon_G(v).
\]
It has been studied extensively, see  \cite{1,2,3,4,6}.

Gupta et al. in 2000 \cite{GSM},  proposed a new topological index involving degree and eccentricity called the connective eccentricity index,  defined as
\[
\xi^{ce}(G)=\sum_{v \in V(G)}\frac{d_G(v)}{\varepsilon_G(v)}.
\]
From experiments for treating hypertension of chemical compounds like non-peptide N-benzylimidazole derivatives, the results obtained using the connective eccentricity index were better than the corresponding values obtained using Balaban’s mean square distance index. Therefore it is worth studying mathematical properties of connective eccentricity index.

Mathematical properties of connective eccentricity index have been studied extensively for trees, unicyclic and general graph. In particular, Yu and Feng \cite {YF} obtained upper or lower bounds for connective eccentricity index  of graphs in terms of many graph parameters such as radius,  maximum degree,
independence number, vertex connectivity, minimum degree, number of pendant vertices and number of cut edges. Li and Zhao \cite {LZ} studied the extremal properties of connective eccentricity index among n-vertex trees with given graph parameters such as, number of pendant vertices, matching number, domination number, diameter, vertex bipartition. Xu et al. \cite {XDL} characterized the extremal graph for connective eccentricity index among all connected graph with fixed order and fixed matching number. 
 For more studies on connective eccentricity index of graphs we refer \cite{LLZ, TWLF, YQTF} and the references cited therein.

In the present paper, as a continuance we mainly study the mathematical properties of the connective eccentricity index of graphs in terms of various graph invariants. First, we determine the graph which attains the maximum CEI among $n$-vertex graphs with fixed connectivity and diameter, then we identify the unique graph with given connectivity and independence number having the maximum CEI. Finally, we characterize those graph with maximum CEI among $n$-vertex graphs with fixed connectivity and minimum degree.

\begin{lemma}\cite{YF}\label{L1}
  Let $G$ be a graph with a pair of non adjacent vertices $u, v$. Then   $\xi ^{ce}(G) < \xi ^{ce}(G+uv)$.
\end{lemma}

\section{Results}

Let $\mathbb{G}_k(n,d)$ be the class of all $k$-connected graphs of order $n$ with diameter $d$. If $d=1$, then $K_n$ is the unique graph in $\mathbb{G}_k(n,1)$. Therefore, we consider $d\geq 2$ in what follows.

Denoted by
$
G(n,k,d)= K_1 \vee [(d-2)/2]K_k \vee K_{n-kd+2k-2} \vee [(d-2)/2]K_k \vee K_1
$
for even $d \geq 4$, and let
$\mathcal{H}(n,k,d)$ be the set of graphs of $K_1 \vee [(d-3)/2]K_k \vee K_{s+1} \vee K_{t+1} \vee [(d-3)/2]K_k \vee K_1$, where  $s, t \geq k-1$ and  $s+ t= n-kd+3k-4$ for odd $d \geq 3$.

\begin{theorem}\label{T1}
Let $G$ be a graph in $\mathbb{G}_k(n,d)$ with maximum CEI, where $d\geq 3$. Then $G \cong G(n,k,d)$ if $d$ is even, and $G \in \mathcal{H}(n,k,d)$ otherwise.
\end{theorem}

\begin{proof}
Let $G \in \mathbb{G}_k(n,d)$ such that $\xi ^{ce}(G)$ is as large as possible. Let $P:= u_0u_1 \dots u_d$ be a diametral path in $G$. Let $A_i= \{v\in V(G) : d_G(v, u_0)=i\}$. Then $|A_0|=1$ and  $A_0\cup A_1\cup \dots \cup A_d$ is a partition of $V(G)$.

Note that $G$ is a $k$-connected graph, we have $|A_i| \geq k $ for $i\in  \{1,2, \dots , d-1\}$.
By Lemma \ref{L1}, we have   $G[A_i]$ and $G[A_{i-1}\cup A_i]$ are complete graphs for $i\in  \{1, \dots , d\}$.
We claim that $|A_d|=1$; Otherwise, we choose a vertex $v\in A_d\setminus \{u_d\}$ and let $G^\ast =G+\{vx : x\in A_{d-2}\}$. Clearly, $G^\ast \in \mathbb{G}_k(n,d)$. By Lemma \ref{L1}, $\xi ^{ce}(G) < \xi ^{ce}(G^\ast)$, a contradiction. So $|A_d|=1$.
Thus, we have $|A_0|=|A_d|=1$, and $|A_i|\ge k$ for $i\in \{2, \dots, d-1\}$.

\noindent {\bf Case 1}. $d$ is even with $d \geq 4$.  then $|A_1|= |A_2|=\dots = |A_{\frac{d}{2}-1}|= |A_{\frac{d}{2}+1}| =\dots = |A_{d-1}|=k $ and $|A_{\frac{d}{2}}|= n-kd+2k-2$.

First, we claim that  $|A_1|=|A_{d-1}|= k$. Suppose that $|A_1|\geq k+1$, then we choose $w\in A_1\setminus \{u_1\}$ and let $G' =G-wu_0+\{wx : x\in A_3\}$. Clearly, $A_0\cup (A_1\setminus \{w\})\cup (A_2\cup \{w\}) \cup A_0\cup  \dots \cup A_d$ is a partition of $ V(G')$. From the construction of $G'$, we have   $d_G(v) = d_{G'}(v), \varepsilon_G(v) = \varepsilon_{G'}(v)$ for all $v \in V(G)\setminus (A_3\cup \{u_0,w\})$. Moreover,
 \begin{eqnarray*}
   d_G(u_0) &=& d_{G'}(u_0)+1,  \\
   \varepsilon_G(u_0) &=& \varepsilon_{G'}(u_0)=d, \\
   d_G(w) &=& d_{G'}(w)+1-|A_3|, \\
   \varepsilon_G(w) &>& \varepsilon_{G'}(w) , \\
  d_G(x) &=& d_{G'}(x)-1, \\
  \varepsilon_G(x) &=& \varepsilon_{G'}(x)< d \hbox{ for all $x \in A_3$}.
 \end{eqnarray*}
 By the definition of CEI, we have
 \begin{eqnarray*}
    \xi^{ce}(G) - \xi^{ce}(G') &=& \frac{d_G(u_0)}{\varepsilon_G(u_0)} - \frac{d_{G'}(u_0)}{\varepsilon_{G'}(u_0)}
                                    +\frac{d_G(w)}{\varepsilon_G(w)}- \frac{d_{G'}(w)}{\varepsilon_{G'}(w)}\\
                               &&+ \sum_{x \in A_3}\left( \frac{d_G(x)}{\varepsilon_G(x)}-
                               \frac{d_{G'}(x)}{\varepsilon_{G'}(x)}\right)\\
                               &&+ \sum_{v \in V(G)\setminus (A_3\cup \{u_0,w\})}\left( \frac{d_G(v)}{\varepsilon_G(v)}-
                                   \frac{d_{G'}(v)}{\varepsilon_{G'}(v)}\right)\\
                               &=& \frac{1}{d} +\frac{d_G(w)}{\varepsilon_G(w)}- \frac{d_{G'}(w)}{\varepsilon_{G'}(w)}+
                                    |A_3| \left(\frac{-1}{\varepsilon_{G}(x)}\right)\\
                               &<& \frac{1}{d} +\frac{1-|A_3|}{\varepsilon_{G'}(w)}- \frac{|A_3|}{\varepsilon_{G}(x)}\\
                               &\leq& \frac{1}{d}- \frac{|A_3|}{\varepsilon_{G}(x)}\\
                               &<& 0,
\end{eqnarray*}
  where the last inequality follows from the fact that $\varepsilon_{G}(x) < d $ and $|A_3|\ge k\ge 1$. Thus, $\xi^{ce}(G) < \xi^{ce}(G')$, a contradiction to the choice of $G$. Therefore, $|A_1|= k$. Similarly, $|A_{d-1}|=k$, as claimed.
By similar argument as above  we may also show that $|A_2|=|A_{d-2}|=k$, \dots, $|A_{\frac{d}{2}-1}|= |A_{\frac{d}{2}+1}|=k$.
Then we have  $|A_{\frac{d}{2}}|= n-kd+2k-2$.
Therefore,  $G \cong G(n,k,d)$.

\noindent {\bf Case 2.} $d$ is odd with $d \geq 3$. By similar argument as in Case 1, we have $|A_1|= |A_d|=k$, $|A_2|=|A_{d-1}|=k$, \dots,  $|A_{\frac{d-3}{2}}|= |A_{\frac{d+3}{2}}| =k$.  It follows that  $|A_{\frac{d-1}{2}}|+|A_{\frac{d-1}{2}}|= n-kd+3k-2$. Therefore, $G \in \mathcal{H}(n,k,d)$. Now
only need to show that all graphs in  $\mathcal{H}(n,k,d)$ have equal CEI. Let $G_1= K_1 \vee [(d-3)/2]K_k \vee K_{z+1} \vee [(d-3)/2]K_k \vee K_1$, where $z= n-kd+2k-3$. Clearly, $G_1 \in \mathcal{H}(n,d)$. For a graph $G_2= K_1 \vee [(d-3)/2]K_k \vee K_{s+1} \vee K_{t+1} \vee [(d-3)/2]K_k \vee K_1$, we assume its vertex partition  $A_0\cup A_1\cup \dots \cup A_d$ is defined as above. 
 If one of $s,t$ is $k-1$, then $\xi ^{ce}(G_1) \cong \xi ^{ce}(G_2)$. Suppose  that $s,t \geq k$. Let $M\subseteq A_{\frac{d+1}{2}} \setminus \{u_{\frac{d+1}{2}}\}$ and $|M|=t-k+1$. Now we obtain $G_1$ from $G_2$ by the following graph transformation:
 \[
 G_1=G_2-\{xy : x \in M, y \in A_{\frac{d+3}{2}} \} + \{xy : x \in M, y \in A_{\frac{d-3}{2}} \}.
 \]
 Then, it is easy to see that $A_0\cup A_1\cup   \dots \cup  A_{\frac{d-3}{2}} \cup ( A_{\frac{d-1}{2}}\cup M)\cup ( A_{\frac{d+1}{2}}\setminus M)\cup A_{\frac{d+3}{2}} \cup  \dots \cup A_d$ is a partition of $V(G_1)$.
From the construction of $G_1$, we have   $\varepsilon_{G_1}(v) = \varepsilon_{G_2}(v)$ for all $v \in V(G_2)$ and $d_{G_1}(v) = d_{G_2}(v)$  for all $v \in V(G_2)  \setminus (A_{\frac{d+3}{2}} \cup A_{\frac{d-3}{2}})$, it follows that
 \begin{eqnarray*}
  d_{G_2}(x) &=& d_{G_1}(x)+t-k+1 \hbox{ for each $x \in A_{\frac{d+3}{2}} $}, \\
  d_{G_2}(x) &=& d_{G_1}(x)-(t-k+1)\hbox{ for each $x \in A_{\frac{d-3}{2}} $}.
 \end{eqnarray*}
 By the definition of CEI, we have
 \begin{eqnarray*}
    \xi^{ce}(G_2) - \xi^{ce}(G_1) &=& \sum_{x \in A_{\frac{d+3}{2}}}\left( \frac{d_{G_2}(x)}{\varepsilon_{G_2}(x)}
                                    -\frac {d_{G_1}(x)}{\varepsilon_{G_1}(x)}\right)\\
                               &&+ \sum_{x \in A_{\frac{d-3}{2}}}\left( \frac{d_{G_2}(x)}{\varepsilon_{G_2}(x)}- \frac{d_{G_1}(x)}{\varepsilon_{G_1}(x)}\right)\\
                               &&+ \sum_{v \in V(G_2)  \setminus (A_{\frac{d+3}{2}} \cup A_{\frac{d-3}{2}})}\left( \frac{d_{G_2}(v)}{\varepsilon_{G_2}(v)}-
                                   \frac{d_{G_1}(v)}{\varepsilon_{G_1}(v)}\right)\\
                               &=& k\left( \frac{t-k+1}{\varepsilon_{G_2}(x)}
                                    -\frac {t-k+1}{\varepsilon_{G_2}(x)}\right)\\
                               &=&0.
\end{eqnarray*}
Thus,  $\xi^{ce}(G_2) = \xi^{ce}(G_1)$. This completes the proof.
\end{proof}

%
%
%
%
%

Let $\mathbb{G}_k(n,\alpha)$ be the class of all $k$-connected graphs of order $n$ with independence number $\alpha$. If $\alpha=1$, then $K_n$ is the unique graph in $\mathbb{G}_k(n,1)$ with maximum CEI. Therefore, we consider $\alpha\geq 2$ in what follows. Let
\[
S_{n,\alpha }= K_k \vee (K_1 \cup ( K_{n-k-\alpha} \vee (\alpha -1)K_1).
\]
Obviously, $S_{n,\alpha } \in \mathbb{G}_k(n,\alpha)$

\begin{theorem}
 Among all graphs in $\mathbb{G}_k(n,\alpha)$ with $\alpha\geq 2$, $S_{n,\alpha }$  is the unique graph with maximum CEI.
\end{theorem}

\begin{proof}
 Note that since $G \in \mathbb{G}_k(n,\alpha)$, we have $k+\alpha \le n$. If $k+\alpha= n$ then $G \cong K_k \vee \alpha K_1$, and the result holds in this case.  Therefore, we consider the case $k+\alpha +1 \le n$ in what follows.

 Let $G \in \mathbb{G}_k(n,\alpha)$ such that $\xi ^{ce}(G)$ is as large as possible. Let $I$ and $A$ be the maximum  independent set and vertex cut of $G$, respectively with $|I|=\alpha $ and $|A|=k$. Let $G_1, G_2, \dots ,G_s$ be the components of $G-A$ with $s\geq 2$. Assume that  $|G_1|\geq |G_2|\geq \dots \geq|G_s|$. We claim that $G_1$ is non-trivial; Otherwise, then $G_i$ is trivial for $i\in  \{1,2, \dots , s\}$, and the independence number of $G$ is at least $n-k ~(\geq \alpha +1)$, a contradiction. So, $G_1$ is non-trivial. Let $|A\cap I|=a, |A\setminus I|=b$ and $|V(G_i)\cap I|=n_i|, |V(G_i)\setminus I|=m_i$ for $i\in  \{1,2, \dots , s\}$. Obviously, $k=a+b$ and $V(G_i)=n_i+m_i$ for $i\in  \{1,2, \dots , s\}$.  We proceed with the following claims.

 \noindent {\bf Claim 1.} $G-A$ contains exactly two components, i.e., $s=2.$

\noindent {\bf Proof of Claim 1.}
 Suppose that $s\geq 3$. Since $G_1$ is non-trivial, we have $V(G_1)\setminus I \neq \emptyset$. Then choose $u \in V(G_1)\setminus I$ and $v \in V(G_2)$. Let $H=G+uv$. Clearly, $H \in \mathbb{G}_k(n,\alpha)$.  By Lemma \ref{L1}, we have $\xi ^{ce}(G) < \xi ^{ce}(H)$, a contradiction. So, $s=2.$

 \noindent {\bf Claim 2.} $G[A] \cong K_b \vee aK_1, G_i \cong K_{m_i}\vee n_iK_1$ and $G[V(G_i)UA] \cong K_{b+m_i}\vee (a+n_i)K_1$ for $i=1,2$.

\noindent {\bf Proof of Claim 2.}
 First, we show that $G[A] \cong K_b \vee aK_1$. Suppose that $G[A] \ncong K_b \vee aK_1$. Then there exist $u,v \in A \setminus I$ or $u \in A \setminus I, v \in A\cap I$. Let  $Q=G+uv$. Clearly, $Q \in \mathbb{G}_k(n,\alpha)$. By Lemma \ref{L1}, we have $\xi ^{ce}(G) < \xi ^{ce}(Q)$, a contradiction. So, $G[A] \cong K_b \vee aK_1$. By similar techniques we can show that $G_i \cong K_{m_i}\vee n_iK_1$ and $G[V(G_i)UA] \cong K_{b+m_i}\vee (a+n_i)K_1$ for $i=1,2$.

  \noindent {\bf Claim 3.} $G_2$ is trivial.

  \noindent {\bf Proof of Claim 3.}
   Suppose that $G_2$ is non-trivial. Then  we have the following two possible cases.

\noindent  {\bf Case 1.} $n_2=0$.
If $a=0$, then $I=V(G_1)\cap I$. Choose $w \in V(G_2)\setminus I$, we get $I\cup \{w\}$ is an independent set such that $|I\cup \{w\}|=\alpha +1$, a contradiction. So, $a \geq 1$.
   Let $G' =G-\{wx : x \in V(G_2)\setminus \{w\}\}+ \{xy : x \in V(G_2)\setminus \{w\}, y\in V(G_1)\}$.
   Clearly, $G' \in \mathbb{G}_k(n,\alpha)$. From the construction of $G'$, we have $\varepsilon_G(v) = \varepsilon_{G'}(v)=1$ for each $v \in A \setminus I$ and   $\varepsilon_G(v) = \varepsilon_{G'}(v)=2$ for each $v \in V(G)\setminus (A \setminus I)$. Moreover,
  \begin{eqnarray*}
   d_G(w) &=& d_{G'}(w)+m_2-1,  \\
   \varepsilon_G(w) &=& \varepsilon_{G'}(w)=2,\\
   d_G(x) &=& d_{G'}(x)+1- n_1- m_1 \hbox{ for all $x \in V(G_2)\setminus \{w\}$}, \\
   \varepsilon_G(x) &=& \varepsilon_{G'}(x)=2 \hbox{ for all $x \in V(G_2)\setminus \{w\}$}, \\
   d_G(x) &=& d_{G'}(x)+1- m_2 \hbox{ for all $x \in V(G_1)$},\\
   \varepsilon_G(x) &=& \varepsilon_{G'}(x)=2 \hbox{ for all $x \in V(G_1)$},\\
   d_G(x) &=& d_{G'}(x) \hbox{ for all $x \in A$}.
 \end{eqnarray*}

  By the definition of CEI, we have

 \begin{eqnarray*}
    \xi^{ce}(G) - \xi^{ce}(G') &=& \frac{d_G(w)}{\varepsilon_G(w)} - \frac{d_{G'}(w)}{\varepsilon_{G'}(w)}
                                    +\sum_{x \in V(G_2)\setminus \{w\}}\left( \frac{d_G(x)}{\varepsilon_G(x)}- \frac{d_{G'}(x)}{\varepsilon_{G'}(x)}\right)\\
                               &&+ \sum_{x \in V(G_1)}\left( \frac{d_G(x)}{\varepsilon_G(x)}- \frac{d_{G'}(x)}{\varepsilon_{G'}(x)}\right)\\
                               &&+ \sum_{x \in A}\left( \frac{d_G(x)}{\varepsilon_G(x)}- \frac{d_{G'}(x)}{\varepsilon_{G'}(x)}\right)\\
                               &=& \frac{1}{2}{m_2-1-(m_2+n_2-1)^2-(m_2-1)(m_1+n_1-1) }\\
                               &<&0,
\end{eqnarray*}
  a contradiction.

\noindent  {\bf Case 2.} $n_2 \neq 0$.

  Choose $w \in V(G_2)\cap I$. Let
  $G'' =G-\{wx : x \in V(G_2)\}+ \{xy : x \in V(G_1) \cap I , y\in V(G_2)\setminus I\}+ \{xy : x \in V(G_1) \setminus I , y \in V(G_2)\setminus \{w\}\}
  $.
   Clearly, $G'' \in \mathbb{G}_k(n,\alpha)$. From the construction of $G''$, we have $\varepsilon_G(v) = \varepsilon_{G''}(v)$ for all $v \in V(G)$. Moreover,
  \begin{eqnarray*}
   d_G(v) &=& d_{G''}(v) \hbox{ for all $v \in A$}, \\
   d_G(w) &=& d_{G''}(w)+m_2,\\
   d_G(x) &=& d_{G''}(x)+1-n_1-m_1 \hbox{ for all $x \in V(G_2)\setminus I$}, \\
   d_G(x) &=& d_{G''}(x)-m_1 \hbox{ for all $x \in (V(G_2)\cap I)\setminus \{w\}$}, \\
   d_G(x) &=& d_{G''}(x)-m_2 \hbox{ for all $x \in V(G_1)\cap I$}, \\
   d_G(x) &=& d_{G''}(x)+1-n_2-m_2 \hbox{ for all $x \in V(G_1)\setminus I$}. \\
 \end{eqnarray*}
By the definition of CEI, we have

 \begin{eqnarray*}
    \xi^{ce}(G) - \xi^{ce}(G'') &=& \frac{d_G(w)}{\varepsilon_G(w)} - \frac{d_{G''}(w)}{\varepsilon_{G''}(w)}+
                                    \sum_{x \in V(G_2)\setminus I}\left( \frac{d_G(x)}{\varepsilon_G(x)}- \frac{d_{G''}(x)}{\varepsilon_{G''}(x)}\right)\\
                                &&+ \sum_{x \in (V(G_2)\cap I)\setminus \{w\}}\left( \frac{d_G(x)}{\varepsilon_G(x)}-
                                    \frac{d_{G''}(x)} {\varepsilon_{G''}(x)}\right)\\
                                &&+ \sum_{x \in  V(G_1)\cap I}\left( \frac{d_G(x)}{\varepsilon_G(x)}- \frac{d_{G''}(x)}{\varepsilon_{G''}(x)}\right)\\
                                &&+ \sum_{x \in V(G_1)\setminus I}\left( \frac{d_G(x)}{\varepsilon_G(x)}-\frac{d_{G''}(x)}{\varepsilon_{G''}(x)}\right)\\
                                &&+ \sum_{v \in A}\left( \frac{d_G(v)}{\varepsilon_G(v)}-\frac{d_{G''}(v)}{\varepsilon_{G''}(v)}\right)\\
                                &=& \frac{1}{\varepsilon_G(v)} \{ m_2-m_2(n_1+m_1-1)-m_1(n_2-1)\\
                                &-& n_1m - m_1(n_2+m_2-1) \}\\
                                &<&0,
\end{eqnarray*}
 a contradiction. So, $G_2$ is trivial.

  \noindent {\bf Claim 4.} $V(G_2) \subseteq I$.

  \noindent {\bf Proof of Claim 4.}
  Suppose that  $V(G_2) \nsubseteq I$, then   $a \geq 2$. Suppose that $a\leq 1$. If $a=1$, i.e.,  $A\cap I =\{w_1\}$. Let
   $G^\ast =G+\{w_1x : x \in V(G_1)\cap I\}$.
   Clearly, $G^\ast \in \mathbb{G}_k(n,\alpha)$. By Lemma \ref{L1}, we have $\xi ^{ce}(G) < \xi ^{ce}(G^\ast)$, a contradiction. If $a=0$,  then $I\cup V(G_2)$ is an independent set of $G$ such that $|I\cup V(G_2)|=\alpha +1$, a contradiction.   So, $a \geq 2$.
    Since $G_1$ is non-trivial then $V(G_1)\setminus I \neq \emptyset $. Choose $w_2 \in V(G_1)\setminus I$. Let
    $G^{\ast\ast} =G-\{w_1v : v = V(G_2)\}+ \{w_2v : v = V(G_2)\}$.
       Clearly, $G^{\ast\ast} \in \mathbb{G}_k(n,\alpha)$.

     From the construction of $G^{\ast\ast}$, we have $ \varepsilon_G(x) = \varepsilon_{G^{\ast\ast}}(x)=1$ for all $x \in A\setminus I$, $ \varepsilon_G(x) = \varepsilon_{G^{\ast\ast}}(x)=2$ for all $x \in (A\cap I)\cup(V(G_1)\setminus \{w_2\})\cup\{v\}$ , and  $d_G(x) = d_{G^{\ast\ast}}(x)$ for all $x \in V(G)\setminus  \{w_1,w_2\}$. Moreover,
     \begin{eqnarray*}
   d_G(w_1) &=& d_{G^{\ast\ast}}(w_1)+1,  \\
   \varepsilon_G(w_1) &=& \varepsilon_{G^{\ast\ast}}(w_1)=2,\\
   d_G(w_2) &=& d_{G^{\ast\ast}}(w_2)-1, \\
   \varepsilon_G(w_2)=2 &>& \varepsilon_{G^{\ast\ast}}(w_2)=1. \\
  \end{eqnarray*}
   By the definition of CEI, we have
   \begin{eqnarray*}
    \xi^{ce}(G) - \xi^{ce}(G^{\ast\ast}) &=& \frac{d_G(w_1)}{\varepsilon_G(w_1)} - \frac{d_{G^{\ast\ast}}(w_1)}{\varepsilon_{G^{\ast\ast}}(w_1)}
                                            +\frac{d_G(w_2)}{\varepsilon_G(w_2)}- \frac{d_{G^{\ast\ast}}(w_2)}{\varepsilon_{G^{\ast\ast}}(w_2)}\\
                                         &&+ \sum_{x \in V(G)\setminus  \{w_1,w_2\}}\left( \frac{d_G(x)}{\varepsilon_G(x)}- \frac{d_{G^{\ast\ast}}(x)}{\varepsilon_{G^{\ast\ast}}(x)}\right)\\
                                         &=& -\frac{1+d(w_2)}{2}\\
                                         &<&0,
\end{eqnarray*}
   a contradiction. So, $V(G_2) \subseteq I$. From claim 1--4, we have $G \cong S_{n,\alpha }$. This completes the proof.
  \end{proof}

Let $\mathbb{G}_k(n,\delta)$ be the class of all $k$-connected graphs of order $n$ with minimum degree at least $\delta$.  Let
\[
M_{n,\delta }= K_k \vee (K_{\delta-k+1} \cup  K_{n-\delta-1}).
\]
Obviously, $M_{n,\delta } \in \mathbb{G}_k(n,\delta)$

\begin{theorem}
 Among all graphs in $\mathbb{G}_k(n,\delta)$,  $M_{n,\delta }$  is the unique graph with maximum CEI.
\end{theorem}

\begin{proof}
 Note that $G \in \mathbb{G}_k(n,\delta)$, we have $k+1 \le n$. If $k+1= n$, then $G \cong K_k \cong M_{n,\delta }$, and the result holds in this case.  Therefore, we consider the case $k+2 \le n$ in what follows.

 Let $G \in \mathbb{G}_k(n,\delta)$ such that $\xi ^{ce}(G)$ is as large as possible. Let $A$ be the vertex cut of $G$ with $|A|=k$. Let $G_1, G_2, \dots ,G_r$ be the components of $G-A$ with $r\geq 2$. We claim that $r=2$; Otherwise $r\geq 3$, and let $G_1, G_2, G_3$ be at least three components of $G$. Let $G' =G+\{xy : x \in V(G_2), y \in V(G_3) \}$. Clearly, $G' \in \mathbb{G}_k(n,\delta)$.  By Lemma \ref{L1}, we have $\xi ^{ce}(G) < \xi ^{ce}(G')$, a contradiction. So, $r=2.$ Also by Lemma \ref{L1},    $G[V(G_1)\cup A] $ and $G[V(G_2)\cup A] $ are complete. Thus, we have $G\cong K_k \vee (K_{a_1} \cup K_{a_2})$, where  $a_1= |V(G_1)|, a_2= |V(G_2)|$, and $a_1 + a_2 = n-k.$
  Without loss of generality, we assume that $a_1 \leq a_2$.

To complete the proof it sufficies to show that $a_1= \delta -k +1$. Suppose that $a_1 > \delta -k +1$. For $w \in V(G_1)$, let $G' =G-\{wx : x \in V(G_1)\setminus \{w\}\}+\{wx : x \in V(G_2)\} $. Clearly, $G' \in \mathbb{G}_k(n,\delta)$.

   From the construction of $G'$, we have
     \begin{eqnarray*}
   d_G(w) &=& d_{G'}(w)+a_1-a_2-1,  \\
   \varepsilon_G(w) &=& \varepsilon_{G'}(w)=2,\\
   d_G(z) &=& d_{G^{\ast\ast}}(z)+1,  \\
   \varepsilon_G(z) &=& \varepsilon_{G'}(z)=2 \hbox{ for all $z \in V(G_1)\setminus \{w\}$}, \\
   d_G(t) &=& d_{G^{\ast\ast}}(t)-1,  \\
   \varepsilon_G(t) &=& \varepsilon_{G'}(t)=2 \hbox{ for all $t \in V(G_2)$}, \\
   d_G(x) &=& d_{G^{\ast\ast}}(x),\\
   \varepsilon_G(x) &=& \varepsilon_{G'}(x)=1 \hbox{ for all $x \in A$}.
  \end{eqnarray*}

   By the definition of CEI, we have

 \begin{eqnarray*}
    \xi^{ce}(G) - \xi^{ce}(G') &=& \frac{d_G(w)}{\varepsilon_G(w)} - \frac{d_{G'}(w)}{\varepsilon_{G'}(w)}+
                                    \sum_{z \in V(G_1)\setminus \{w\}}\left( \frac{d_G(z)}{\varepsilon_G(z)}- \frac{d_{G'}(z)}{\varepsilon_{G'}(z)}\right)\\
                                &&+ \sum_{t \in V(G_2)}\left( \frac{d_G(t)}{\varepsilon_G(t)}-
                                    \frac{d_{G'}(t)} {\varepsilon_{G'}(t)}\right)\\
                                &&+ \sum_{x \in A}\left( \frac{d_G(x)}{\varepsilon_G(x)}- \frac{d_{G'}(x)}{\varepsilon_{G'}(x)}\right)\\
                                &=&  a_1-1-a_2\\
                                &<&0,
\end{eqnarray*}
where the last inequality follows due to the fact that $a_1 \leq a_2$, a contradiction. So, $a_1= \delta -k +1$, one has $a_2= n-\delta-1$.   Thus, $G \cong M_{n,\delta }$. This completes the proof.
  \end{proof}

\end{document}